\documentclass{amsart}
\usepackage[T1]{fontenc}   
\usepackage{amssymb, graphicx, color}
\usepackage{enumerate, float}
\newcommand{\N}{{\mathbb N}}
\newcommand{\R}{{\mathbb R}}

\newcommand{\LL}{{\mathcal L}}

\newcommand{\PP}{{\mathcal P}}

\newcommand{\px}{{2^{\PP(X)}}}

\newcommand{\bigmeet}{{\bigwedge}}

\newtheorem{theorem}{Theorem}[section]
\newtheorem{corollary}[theorem]{Corollary}
\newtheorem{lemma}[theorem]{Lemma}

\theoremstyle{definition}
\newtheorem{definition}[theorem]{Definition}
\newtheorem{example}[theorem]{Example}

\numberwithin{equation}{section}

\begin{document}

\vspace{0.5in}

\renewcommand{\bf}{\bfseries}
\renewcommand{\sc}{\scshape}
\vspace{0.5in}

\title{Metric axioms: a structural study}

\author{J. Bruno}
\address{National University of Ireland, Galway}
\email{brujo.email@gmail.com}

\author{I. Weiss}
\address{University of South Pacific, Suva, Fiji}
\email{weiss\text{\_}i@usp.ac.fj}





\maketitle

\begin{center} \textbf{Abstract}
\end{center}
\noindent
{\small For a fixed set $X$, an arbitrary \textit{weight structure} $d \in [0,\infty]^{X \times X}$ can be interpreted as a distance assignment between pairs of points on $X$. Restrictions (i.e., \textit{metric axioms}) on the behaviour of any such $d$ naturally arise, such as separation, triangle inequality and symmetry. We present an order-theoretic investigation of various collections of weight structures, as naturally occurring subsets of $[0,\infty]^{X \times X}$ satisfying certain metric axioms. Furthermore, we exploit the categorical notion of adjunctions when investigating connections between the above collections of weight structures. As a corollary, we present several lattice-embeddability theorems on a well-known collection of weight structures on $X$.
 }

\section{Introduction}

For a fixed set $X$, a \textbf{standard metric} on it is any $d \in [0,\infty]^{X\times X}$ for which:

\begin{enumerate}[(i)]
\item $\forall x \in X$, $d(x,x) = 0$
\smallskip
\item $\forall x,y \in X$, $d(x,y) = 0 = d(y,x) \Rightarrow x = y$ (\emph{separation})
\smallskip
\item $\forall x,y \in X$, $d(x,y) = d(y,x)$ (\emph{symmetry})
\smallskip
\item $\forall x,y,z \in X$, $d(x,y) + d(y,z) \geq d(x,z)$ (\emph{triangle inequality}).\\
\end{enumerate}

The collection of all standard metrics on $X$ can then be identified with a particular subset of $ [0,\infty]^{X\times X}$. For convenience, we shall refer to axioms (i), (ii),  (iii) and (iv) as $0, s, \Sigma$ and $\Delta$ respectively. By letting $P$ denote any collection of these axioms and $W_P(X) = \{d \in  [0,\infty]^{X\times X} \mid d \text{ satisfies all axioms in } P\}$, it is also possible to view any such collection of axioms as a subset of the ambient set. For convenience, we suppress $X$ from this notation where there is no danger of ambiguity.  In particular, $W_\emptyset = [0,\infty]^{X\times X}$. Clearly, for $P, Q$ (two collections of axioms) one has that $W_{P \cap Q} \supseteq W_Q \cup W_P$ and $W_{P \cup Q} = W_Q \cap W_P$.  For simplicity, we suppress the use of brackets where $P$ is used as a subscript (e.g., if $P=\{0\}$ then $W_P =W_0$). We refer to any $d \in  [0,\infty]^{X\times X}$ as a \textit{weight structure}.
For reasons that will become apparent in the sequel, this paper is concerned with weight structures that satisfy axiom 0.

The ambient set $ [0,\infty]^{X\times X}$ when ordered pointwise (i.e., $d\leq m \in  [0,\infty]^{X\times X} \Leftrightarrow \forall x,y \in X$, $d(x,y) \leq m(x,y)$) forms a complete lattice; every collection $W_P$ becomes a poset and the obvious inclusions $W_P \hookrightarrow W_Q$ (for $Q\subseteq P$) acquire more structure (that of order-preserving functions). The first part of Section~\ref{sec:met(X)} is concerned with the lattice-theoretic structure of the above defined collections. In particular, we explore a connection between certain first-order predicates and the order-theoretic collection of weight structures they define. The rest of Section~\ref{sec:met(X)} presents a categorical investigation of the inclusions $W_P \hookrightarrow W_Q$. For instance, given any weight structure that does not satisfy symmetry, it is then natural to ask:
\smallskip

\textit{ Is there a way to naturally symmetrize such an element? Is such a process unique? If not, is there an optimal one?}
\smallskip

Expressed in the language of categories, we provide answers to the above by means of adjunctions in much the same spirit as how the process of turning a base into a topology can be seen as an adjunction. A comprehensive diagrams of such adjunctions can be found in Figure~\ref{allads} where for $M \in \{0,s,\Sigma,\Delta\}$ and $M\not \in P \subset \{0,s,\Sigma,\Delta\}$ the symbols $M_*$ (resp. $M_!$) denote the right (resp. left) adjoint to the inclusion $W_{P\cup M} \hookrightarrow W_P$.
\begin{figure}\label{allads}
\vspace{.2cm}
\centering 
\def\svgwidth{200pt} 
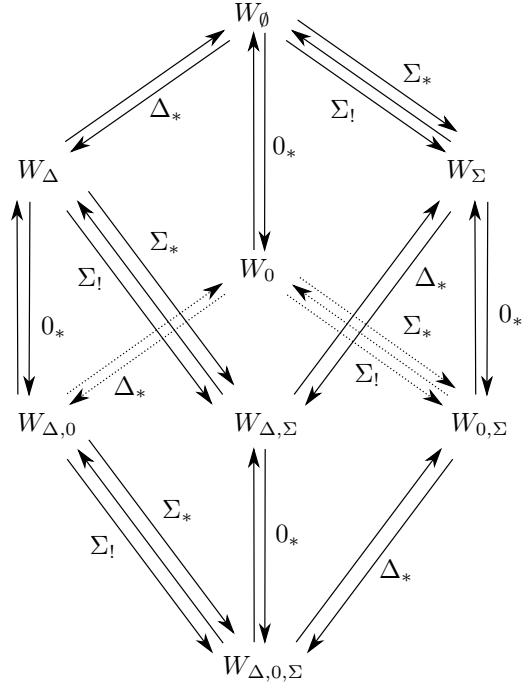
\caption{Adjunctions}
 \end{figure}

For $A\subseteq X$ and $d \in W_P$, $A$ is {\it left $P$-open} (resp. {\it right $P$-open}) provided that for any $x \in A$ one can find $\epsilon > 0$ so that $B_L(x)_\epsilon = \{y \in X \mid d(x,y) < \epsilon\} \subseteq A$ (resp. $B_R(x)_\epsilon = \{y \in X \mid d(y,x) < \epsilon\} \subseteq A$). A standard way to generate a topology from any $d \in W_P$ is given by: $O \in \tau_L$ (resp. $O \in \tau_R$) iff $O$ is left $P$-open (resp.  $O$ is right $P$-open). In other words,

$$ d \mapsto ( \tau_L, \tau_R).$$

\smallskip

A different way to generate topologies from any $d\in W_P$ is by mapping $d$ to the covers its right and left $\epsilon$-balls ($B_R(x)_\epsilon$ and $B_L(x)_\epsilon$, respectively) generate. From there one considers such collections as subbases. More precisely,
$$d \mapsto (\langle\{B_L(x)_\epsilon \mid \epsilon > 0\}\rangle,\langle \{B_L(x)_\epsilon \mid \epsilon > 0\}\rangle).
$$

Notice that both methods coincided if, and only if, $\Delta \in P$. Since $Top(X)$ (the collection of all topologies on a fixed set $X$) is a complete lattice (see \cite{MR3003323} and \cite{MR2944774}), the above assignments can be interpreted as order-theoretic functions $W_P \to Top(X) \times Top(X)$. Section~\ref{sec:propsi} is devoted to a detailed investigation of the above functions with particular emphasis on meet/join preservation, structure of fibers and order-preservation.

Lastly, Section~\ref{sec:metx} is concerned with the lattice-theoretic structure of the collection of all extended metrics $Met(X)= W_{0,\Delta,\Sigma}$ on a fixed set $X$. We explore Menger convexity (and its dual) within $Met(X)$ along with several lattice-embeddability properties of $Met(X)$.

\section{Preliminaries}

For a partial order $\mathbb{P}$, we adopt the name \emph{meet semilattice} (resp. \emph{join semilattice}) whenever $\mathbb{P}$ is closed under all finite meets (resp. finite joins). In particular, if $\mathbb{P}$ is a meet (resp. join) semilattice, then the empty meet (resp. join) must exist within $\mathbb{P}$; vacuously, $\bigvee \emptyset = \bot$ and $\bigwedge \emptyset =\top$ and, consequently, meet and join semilattices have top and bottom elements, respectively. We will refer to $\mathbb{P}$ as a \emph{lattice} provided it is a meet and join semilattice. Lastly, $\mathbb{P}$ is a \emph{complete lattice} whenever it contains all meets and joins. Following the above definitions, it is not necessary to distinguish join and meet completeness since both notions yield complete lattices. That is, if P is a complete join semilattice (i.e., it has all joins), then it also has all meets, and thus is complete. For lattices $\mathbb{P}$ and $\mathbb{Q}$, an order-preserving map $f:\mathbb{P}\to \mathbb{Q}$ is said to \emph{preserve meets} or be \textit{closed under meets }(resp. \emph{preserve joins} or \textit{closed under joins}) iff for any pair $x,y \in \mathbb{P}$ we have $f(x \wedge y) = f(x) \wedge f(y)$ (resp. $f(x \vee y) = f(x) \wedge f(y)$). Whenever $f$ is injective and preserves all joins and all meets, $f$ is said to be an \emph{embedding} of $\mathbb{P}$ into $\mathbb{Q}$ and that $\mathbb{P}$ can be \emph{embedded} within $\mathbb{Q}$. Given a function $g:X\to Y$ ($X$ and $Y$ sets) we adopt the standard meaning of fibers of functions; for $y \in Y$ the \emph{fiber of $g$ over $y$}  is just $\{x\in X \mid g(x) = y\} = g^{-1}(y)$. For $\mathbb{P}$ a complete lattice and $A \subset \mathbb{P}$, $A$ is a \emph{sublattice} (resp. \emph{complete sublattice}, \emph{complete join sublattice},  \emph{complete meet sublattice}) if the inclusion function is closed under all finite meets and joins (resp. closed under all meets and joins, closed under all joins, closed under all meets). We note that the latter four are different notions.

All of our basic terminology regarding categories is standard and can be found in \cite{MR1712872}. For convenience, we will recall that given functors $F: \textbf{C} \to \textbf{D}$ and $G: \textbf{D} \to \textbf{C}$, $F$ is said to be \emph{left adjoint} to $G$ (and $G$ is \emph{right adjoint} to $F$) provided that there exists a natural bijection (in the variables $X$ and $Y$) between morphisms $f:X \to G(Y)$ in $\textbf{C}$ and $\overline{f}:F(X) \to Y$ in $\textbf{D}$. The pair $F,G$ is commonly referred to as an \emph{adjunction} and we denote the adjunction by $G:D\leftrightarrows C:F$. In the sequel we frequently describe posets as categories where arrows point in the direction of the order. In other words, for $a,b \in \mathbb{P}$ (a poset) we have $a \to b$ iff $a\leq b$ and order-preserving maps become functors.

For a first-order language $\LL$, formula $\phi(x_,\ldots,x_n)$ from $\LL$ and structure $M$, the $n$-tuple $\langle a_1, \ldots , a_n\rangle \in M^n$ is said to {\it satisfy} $\phi(x_,\ldots,x_n)$ provided
$$ \models_M \phi(a_1, \ldots, a_n).$$

The set $\phi_M = \{\langle a_1, \ldots , a_n\rangle \mid  {\models_M\phi(a_1, \ldots, a_n)}\}$ is said to be {\it defined} by $\phi(x_1, \ldots, x_n)$. If $\phi_M = N^n$ for some $N\subseteq M$, then $N$ is also said to be {\it defined} by $\phi(x_1,\ldots,x_n)$.

\section{Substructures of $[0,\infty]^{X \times X}$}\label{sec:met(X)}

We begin this section by investigating the hierarchy of weight structures $W_P$ in the lattice $[0,\infty]^{X \times X}$, exploring their lattice-theoretic and topological properties and connections between the two.

Elements of $Met(X)$ are called \emph{metric structures} on $X$ while $Top_M(X)$ will denote the collection of topologies generated by elements from $Met(X)$.

\subsection{Compactness and Completeness within $[0,\infty]^{ X\times X}$}

As a topological space, we consider $[0,\infty]$ as the one-point compactification of $[0,\infty)$. It follows that the product space $[0,\infty]^{ X\times X}$ is compact and Hausdorff, and compact sets
and closed sets coincide within $[0,\infty]^{ X\times X}$. Given an open subset $O$ of $[0,\infty]$, and $a \in X$, we denote by $O_{a}$
the subbasic open set consisting of all functions whose evaluation of $a$
belongs to $O$. Recall that the interval topology on a lattice $L$ is the one generated
by the rays $x^{\uparrow}=\{y\in L\mid x\leq y\}$, $x^{\downarrow}=\{y\in L\mid x\geq y\}$
(for all $x\in L$) as subbasic closed sets.

\begin{theorem} The product topology on
$[0,\infty]^{ X\times X}$ is exactly the same as the interval topology on
it. Moreover, any sublattice $L$ of $[0,\infty]^{ X\times X}$ is compact iff $L$ is closed iff $L$ is complete.
\end{theorem}

\begin{proof}
Notice that rays are closed in the product space $[0,\infty]^{ X\times X}$ for if we let $f\in[0,\infty]^{ X\times X}$
and take any $h\in [0,\infty]^{ X\times X}\smallsetminus f^{\uparrow}$, then
for some $x\in X$ we have that $h(x)<f(x)$ and we can create an
$\epsilon-$interval around $h(x)$ so that $h(x)-\epsilon<h(x)+\epsilon<f(x)$.
It follows that $(h(x)-\epsilon,h(x)+\epsilon)_{x}\cap f^{\uparrow}=\emptyset$,
where $(h(x)-\epsilon,h(x)+\epsilon)_{x}=\{g \in[0,\infty]^{ X\times X}\mid g(x)\in(h(x)-\epsilon,h(x)+\epsilon)\}$
is a basic open set, and so $f^{\uparrow}$ is closed in the product topology. A similar argument shows that $f^{\downarrow}$ is closed in the product topology. The other inclusion works as follows: take a subbasic open set $O_{a}$ in the product topology (where we can assume
$O=[0,d)$ or $(c,\infty]$). Since the collection of rays (as explained above) generates
all closed sets in the interval topology, the collection of complements
of rays generates all open sets. With that in mind, consider $f_{c},f_{d}\in [0,\infty]^{ X\times X}$
so that:

\[
f_{c}(x)=\begin{cases}
c & x=a\\
0 & \mbox{otherwise.}
\end{cases}
\]

\[
f_{d}(x)=\begin{cases}
d & x=a\\
\infty & \mbox{otherwise.}
\end{cases}
\]

\medskip

The claim is that $[0,c)_{a}= [0,\infty]^{ X\times X}\smallsetminus f_{c}^{\uparrow}$ and $(d,\infty]_a = [0,\infty]^{ X\times X}\smallsetminus f_{d}^{\downarrow}$.
Clearly, $[0,c)_{a}\subseteq (f_{c}^{\uparrow})^{C}$ and $(d,\infty]_a \subseteq (f_{d}^{\downarrow})^{C}$.
If $h\in(f_{c}^{\uparrow})^{C}$ and since $h\not\geq f_{c}$, then
 it must be that $h(a)<f_{c}(a)=c$ (since $h(x)\geq f_{c}(x)$
for all other $x\not=a$) and thus equality follows. The same happens with $f_{d}$. The second claim is due in part to Frink (\cite{MR0006496}, where he shows that any lattice is compact in its interval topology if, and only if, it is complete) and, partly, to the subspace topology on a complete sublattice of $[0,\infty]^{ X\times X}$ being the same as the interval topology. Indeed, take any $f \in [0,\infty]^{ X\times X}$ and a complete sublattice $L$; since $L$ is complete, $\bigmeet \{ h \in L \mid h \geq f\} = g$ exists and thus $g^{\uparrow}\cap L = f^{\uparrow}\cap L$.

\end{proof}

Since $W_\emptyset = [0,\infty]^{ X\times X}$ we will use $W_\emptyset$ to denote our ambient structure $[0,\infty]^{ X\times X}$. Clearly, $W_0$ and $W_\Sigma$ are sublattices of $W_\emptyset$, $W_s$ is closed under non-empty meets and joins from $W_\emptyset$, and $W_\Delta$ is closed under arbitrary joins from $W_\emptyset$. If we consider two weight structures $d,d'\in Met(X)\cap W_s$ so that for a triple $x,y,z\in X$, $d(x,z)=d'(x,z)$ and $d(x,y)=d'(y,z)<\frac{d(x,z)}{2}$, then $d\wedge d'\not \in W_\Delta$. Consequently, neither $Met(X)$ nor $W_\Delta$ is closed under finite meets. Notice, also, that for any $d \in W_\emptyset$ there is a unique $d'\in W_\emptyset$ for which $d(x,y) = d'(y,x)$ and we refer to $d'$ as the {\it dual} of $d$. Of particular importance is to notice that for any $W_P$, the function $f:W_P \to W_P$ for which $d \mapsto d'$ is an order isomorphism with $f^2 = id$.

Let $\mathcal{{L}}$ be a first-order language on \textup{$X\times[0,\infty]$} (i.e., Boolean algebras, lattices, etc) and $\phi(x_1,\ldots, x_n)$ a formula from the language. A set $\mathcal{{F}\subseteq}[0,\infty]^{X}$
is said to be \textit{expressed} in $\mathcal{{L}}$ via $\phi(x_1,\ldots, x_n)$ provided
$$\mathcal{{F}}=\{F\subset X\times[0,\infty] \mid \forall \langle a_1,\ldots, a_n\rangle \in F^n \mbox{, } {\models_F\phi(a_1,\ldots, a_n)}\}.$$
\noindent
In other words, $\mathcal{F}$ is expressed by the sentence $\forall x_1\ldots x_n \phi(x_1, \ldots x_n)$. A formula $\phi(x_1,\ldots, x_n)$ from the language $\mathcal{L}$ is said to be \emph{relaxed} if for any collection $A=\{p_{i}=((x_{i},y_i),a_{i})\}_{0\leq i\leq n}$
(where $x_{i}\in X$ and $a_{i}\in[0,\infty]$) for which $\not \models_A \phi(p_1,\ldots, p_n)$ there exists $\epsilon>0$ so
that for any other collection $B=\{q_{i}=((w_{i},z_i),b_{i})\mid a_{i}-\epsilon<b_{i}<a_{i}+\epsilon\}_{0\leq i\leq n}$
we have that $\not\models_B \phi(q_1,\ldots, q_n)$.

\begin{lemma}
Let $\mathcal{{L}}$ be as above. If a subset of $W_\emptyset$ is expressible by a relaxed quantifier-free formula from $\mathcal{\mathcal{{L}}}$,
then it is compact.
\end{lemma}

\begin{proof}
Let $\phi(x_{1},...,x_{n})$ be a relaxed well-formed atomic formula.
Define $\phi^{\ast}=\forall x_{1},...,x_{n}\phi(x_{1},...,x_{n})$
and let $F=\{f\subseteq X\times[0,\infty] \mid \models_f \phi^{\ast}\}$.
If $g\not\in F$ we can find  $p_{1}=((x_1,y_{1}),a_{1}), \ldots , p_{n}=((x_n,y_{n}),a_{n})\in g$ (since $\phi$ is quantifier free)
so that $\not \models_g\phi(p_{1}, \ldots , p_{n})$. Since $\phi$ is relaxed, there exists an $\epsilon>0$ for which any
$h \subseteq X\times[0,\infty]$ so that $h \supset \{q_{i}=((w_{i},z_i),b_{i})\mid a_{i}-\epsilon<b_{i}<a_{i}+\epsilon\}_{0\leq i\leq n}$ is such that $\not\models_h\phi(q_1,\ldots,q_n)$.
The collection of all such $h$ is a standard basic open set containing
$g$ and disjoint from $F$. Indeed, it can be written as $\bigcap_{i\leq n} (a_i - \epsilon, a_i + \epsilon)_{x_i}$. Hence, $F$ is closed and thus compact.
\end{proof}

\begin{corollary}
Let $\mathcal{{L}}$ be as described above. If a subset of $W_\emptyset$ is expressible by a collection of relaxed quantifier-free formulas from $\mathcal{\mathcal{{L}}}$,
then it is compact.
\end{corollary}

\begin{proof} This is a simple consequence of compact and closed sets being identical in $W_\emptyset$.
\end{proof}

Let $\textbf{P}$ be a unary predicate symbol for $\LL$ for which $\textbf{P}[(x,y),a)]$ is true whenever $x=y$. Further, let {\bf Q} be also unary so that $\textbf{Q}[(x,y),a)]$ is true precisely when $a=0$. By letting $\phi_0(v)$:= ``$\textbf{P}(v)\Rightarrow\textbf{Q}(v)$'' we get that $W_0$ is compact and thus a complete sublattice of $W_\emptyset$.
The same is true for $W_{\Sigma}$ since we can define the
relaxed formula $\phi_{\Sigma}(w,v):=$ ``$w=((x,y),a)$
and $v=((y,x),b)\Rightarrow a=b$'' and let $\overline{\phi_{\Sigma}}:=\forall w,v(\phi_{\Sigma}(w,v))$
be the sentence expressing $W_{\Sigma}$; notice that ``$w=((x,y),a)$
and $v=((y,x),b)$'' and ``$a=b$'' can both be defined by binary predicate symbols. As for $W_{\Delta}$
let $\phi_{\Delta}(u,v,w):=$ ``$u=((x,y),a),v=((y,z),b)$ and $z=((z,x),c)\Rightarrow a+b\geq c$''
and
$$
\overline{{\phi}}_{\Delta}:=\forall x,y,z\phi_{\Delta}(u,v,w)
$$

\smallskip
\noindent
then $\overline{{\phi}}_{\Delta}$ expresses $W_{\Delta}$. Hence, all
combinations thereof are closed and compact within $W_\emptyset$. The same
is not true of $W_{s}$. We can define $\phi_{s}(v)$ to be
``$v=((x,y),0)\Rightarrow x=y$" but this is not a relaxed formula. As a matter of fact $W_{s}$ is not closed in $W_\emptyset$. Indeed, if we take any collection of weight structures $\{d_i\}_{i\in\R^+}\subset W_s$ for which there exists a pair $x,y\in X$ so that $d_i(x,y)=i$ then $\bigwedge d_i\not \in W_s$. That said, $W_s$ is indeed closed under non-empty joins and finite meets from the ambient lattice.

\subsection{Adjunctions}\label{adjunctions}

Consider $2^{\mathcal{P}(X)}=\mathcal{P}(\mathcal{P}(X))$
for a fixed set $X$. Here, meets
are intersections and joins are unions. Next, say that a collection
$C\in2^{\mathcal{P}(X)}$ is a \emph{cover }of $X$ if
\[
\bigvee_{c\in C}c=X.
\]
Define $Cov(X)$ to be the subset of $2^{\mathcal{P}(X)}$ consisting
of the covers of $X$. It is closed under all non-empty joins from $2^{\mathcal{P}(X)}$. The subset of
 $\px$
consisting of the bases for a topology on $X$ will be denoted by $Base(X)$ and $Top(X)$ will denote
the set of all topologies on $X$. The latter is closed under all non-empty meets from $Base(X)$. Since every topology is a base, every base is a cover we get the inclusion sequence $Top(X)\hookrightarrow Base(X) \hookrightarrow Cov(X) \hookrightarrow \px$. It is possible to move back along the previous sequence by, for instance, taking an element from $\px$ and adding $X$ to it to get a cover. This cover is turned into a base by closing it under finite intersection (including the empty intersection) and turning this base into a topology in the usual way. Hence, omitting $(X)$, we get the following diagram\\

\begin{figure}[H] 
\vspace{.2cm}
\centering 
\def\svgwidth{220pt} 
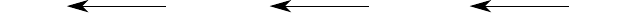
 \end{figure}

 A number of the above constructions can be viewed as adjunctions: for instance, if we consider $Top(X)$ and $Base(X)$ as categories (where arrows point in the direction of the order; $a \rightarrow b$ iff $a\leq b$) then generating a topology from a base amounts to taking a base and mapping it to the meet of all topologies containing it. That is, take $b \in Base(X)$ with the inclusion map $Top(X) \hookrightarrow Base(X)$ and notice that $F:Base(X) \to Top(X)$ so that
$$
\displaystyle b\mapsto \bigwedge_{b\leq a \in Top(X)}
$$
is the functor that generates a topology from a given base (i.e., the left adjoint to the inclusion mapping). This is a special case of the following theorem (cf. \cite{MR861951} pg. 26).

\begin{theorem}\label{thm:johnstone} Let $g: \mathbb{P} \to \mathbb{Q}$ be an order-preserving map between posets. Then
\begin{enumerate}
\item if $g$ has a left adjoint (resp. right adjoint) $f:\mathbb{Q} \to \mathbb{P}$, then $g$ preserves all meets (resp. joins) that exist in $\mathbb{P}$.
\item if $\mathbb{P}$ has all meets (resp. joins) and $g$ preserves them, then $g$ has a left (resp. right) adjoint.
\end{enumerate}
In particular, if $\mathbb{P}$ and $\mathbb{Q}$ are complete lattices then $g$ has a left (resp. right) adjoint iff $g$ preserves all meets (resp. joins).
\end{theorem}

Hence, the mappings $Top(X) \hookrightarrow Base(X)$, $Top(X) \hookrightarrow Cov(X)$ and $Top(X) \hookrightarrow \px$ (resp. $Cov(X) \hookrightarrow \px$) have left adjoints (resp. right adjoint) corresponding to the generation of topologies based on a subbase, base, cover and an arbitrary collection of sets. The inclusion $Base(X) \hookrightarrow Cov(X)$ is not covered by Theorem~\ref{thm:johnstone} since $Base(X)$ has neither all meets nor all joins. That said, this inclusion has a left adjoint: for $c \in Cov(X)$ 
$$ c \mapsto \bigwedge_{c \subseteq b} b.$$

In much the same spirit we can apply the above to the collection of weight structures. For instance,  take any $d\in W_{0}$ and send it to $\displaystyle \bigvee_{d\geq d'\in W_{0,\Delta}}d'$; since $d_0 \in W_{0,\Delta}$ (where $\forall x,y \in X$, $d_0(x,y)=0$) this map is well-defined. This mapping is right adjoint to the inclusion map $W_{0,\Delta}\hookrightarrow W_0$. The results from Section 3.1 in conjunction with Theorem~\ref{thm:johnstone} yield

\begin{figure}[H]
\vspace{.1cm}
\centering 
\def\svgwidth{150pt} 
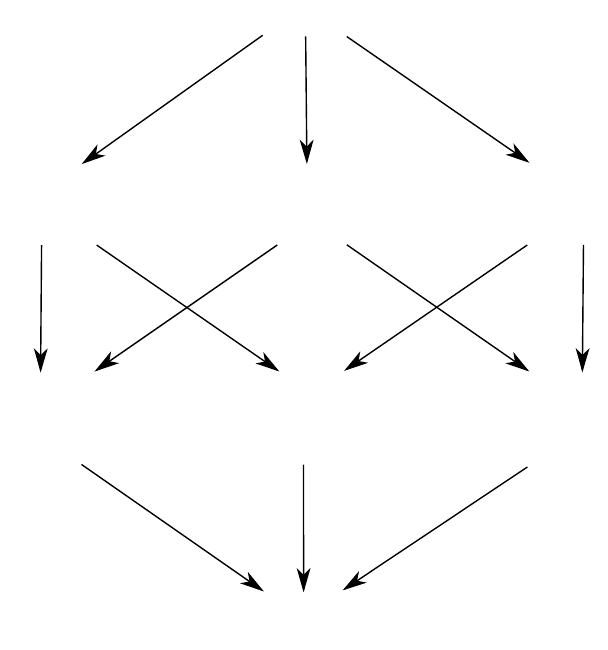 
\vspace{.1cm}
\[
 \text{Right adjoints}
 \]
 \end{figure}

 Notice that the inclusion $W_s \hookrightarrow W_\emptyset$ is not left adjoint and, thus, not included in the above diagram. It's right adjoint, call it $s_*$, would have nowhere to map a weight structure $d\in W_\emptyset$ for which $d(x,y)=0$ and $x\not = y$. Indeed, for any other $d' \in W_s$ we need $d'\leq s_*(d) \Leftrightarrow d'\leq d$ which is impossible. The previous argument can be easily extended to any other inclusion $W_{P\cup s} \hookrightarrow W_P$ with $s\not \in P \subseteq \{\Delta,\Sigma,0\}$. As for left adjoints, the only type of inclusions that are right adjoints are those involving $\Sigma$. More to the point, $W_{P\cup \Sigma} \hookrightarrow W_P$ is right adjoint and only this type of inclusions are right adjoint. Notice that $W_{P \cup \Delta}$ is never closed under meets from $W_P$ and that this fact  settles the cases where $W_{P \cup \Delta} \hookrightarrow W_P$. For $W_s \hookrightarrow W_\emptyset$, notice that for a left adjoint, $s_!$, to exist it must be that if $d\in W_\emptyset$ so that $d(x,y) = 0$ with $x\not = y$, then for any other $d' \in W_s$ we must have $d' \geq s_!(d) \Leftrightarrow d' \geq d$. The meet of all $d'\in W_s$ for which $d' \geq d$ does not exist in $W_s$. Lastly, a left adjoint to $W_{P \cup 0} \hookrightarrow W_P$ does not exist
since the inclusion map does not map the top element in $W_{P\cup0}$ to that of
$W_P$.

 Back to $W_{0,\Delta} \hookrightarrow W_0$, for any $d \in W_0$ that doesn't satisfy the triangle inequality the right adjoints outlined above describe the process for turning $d$ into one that does, call it $m$, and where $d>m$. For a pair $x,y \in X$, $m(x,y)$ must then be the shortest distance, travelling over all finite paths $\gamma:(x=y_0,\ldots, y_n=y)$ on $X$, from $x$ to $y$. For every such path $\gamma$, the distance of each segment $(y_k,y_{k+1})$ will be taken from $d$. In other words, $m(x,y)$ is the infimum of all \textit{lengths} over all finite paths $\gamma$ from $x$ to $y$, where the length of each segment of $\gamma$ is given by $d$. Thus (a) $m<d$, (b) $m$ satisfies the triangle inequality and (c) $m$ is the largest weight structure to satisfy (a) and (b).

\begin{lemma}
For $d\in W_{0}$ and $G:W_{0}\rightarrow W_{0,\Delta}$ given by 

$${\displaystyle d\mapsto\bigvee_{d\geq d'\in W_{0,\Delta}}d'},$$
then for all $x,y\in X$,  $G(d)(x,y) =\bigwedge_{\gamma}\sum_{0}^{n}d(y_{i},y_{i+1})$
where the meet operation ranges over all paths $\gamma:(x=y_{0},y_{1,}\ldots,y_{n}=y)$.

\end{lemma}

\begin{proof}
Let $m(x,y) = \bigwedge_{\gamma}\sum_{0}^{n}d(y_{i},y_{i+1})$ as defined above and note that $m\in W_{0,\Delta}$. Clearly, $m\leq d$ and so $G(d)\geq m$.
If $G(d)>m$ then $\exists x,y\in X$
so that $G(d)(x,y)>m(x,y)=\bigwedge_{\gamma}\sum_{0}^{n}d(y_{i},y_{i+1})$
and consequently, there exists a $\gamma:(x=y_{0},\ldots,y_{n}=y)$
for which $\sum_{0}^{n}G(d)(y_{i},y_{i+1})\geq G(d)(x,y)>\sum_{i}d(y_{i},y_{i+1})$.
To this end, we have that for some $i\leq n$, $G(d)(y_{i},y_{i+1})>d(y_{i},y_{i+1})$
which is a contradiction.
\end{proof}

In Section~\ref{subsec:furems} we present a complete and comprehensive survey of all adjunctions between collections of weight structures (i.e., adjoints of adjoints, etc). In particular, we express the above outlined adjunction's explicitly just as it was done in the previous lemma. 
In \cite{MR0143169} Kelly defines the notion of a pair of \emph{conjugate} pseudo-quasi-metrics on a set $X$; that is, a pair of weight structures $p,q \in W_{0,\Delta}$ for which $p$ and  $q$ are duals of each other (i.e., $q(x,y) = p(y,x)$ for all $x,y \in X$). The set $X$ endowed with the conjugate pair $(X,p,q)$ defines a bitopology on $X$; a set with two associated topologies, one generated by $p$ and the other by $q$ (i.e., generated by their left $\{0,\Delta\}$-open sets). We can take the join of $p$ and $q$ and obtain a symmetric weight structure on $X$ (which Kelly denotes as $d$). The left adjoint to the inclusion $W_{0,\Delta, \Sigma} \hookrightarrow W_{0,\Delta}$  (for which $m \mapsto \bigwedge \{r \in W_{0,\Delta, \Sigma} \mid r\geq m\}$) is the one that takes $p,q \mapsto d$. Clearly the join of the topologies generated by $p$ and $q$ coincides with the one generated by $d$. Hence, the study of bitopological spaces generated by pseudo-quasi-metrics is closely related to that of the adjunction $W_{0,\Delta}\leftrightarrows W_{0,\Delta,\Sigma}$.
 
\subsection{Adjoints: the full picture}\label{subsec:furems}
For $P \subseteq\{s,0,\Delta,\Sigma\}$ and $M$ a property from $\{s,0,\Delta,\Sigma\} \smallsetminus P$, 
let $M^{*}:W_{P\cup M}\hookrightarrow W_{P}$; the collection $P$ will be understood from context and thus its absence from the notation $M^*$.
Its left adjoint, if it exists, will be denoted by $M_{!}$,
while its right adjoint will be denoted by $M_{*}$.

For the property $0$ and any $P$ not containing $0$,
 ${0^{*}:W_{P \cup 0}\to W_{P}}$
has a right adjoint. The right adjoint is given, for $d\in W_{\emptyset}$,
by $0_{*}(d)(x,y)=d(x,y)$ if $x\ne y$, and $0$ otherwise. Clearly,
for any $d'\in W_{0}$ we have that $d'(x,y)\le0_{*}(d)(x,y)\iff0^{*}(d')(x,y)\le d(x,y)$,
establishing the adjunction. Further, $0_{*}$ itself has a right adjoint if and only if $\Delta\not \in P$. Namely,
$\infty_{!}:W_{0}\to W_{\emptyset}$ given by $\infty_{!}(d')(x,y)=d'(x,y)$
if $x\ne y$ and $\infty$ otherwise. Then, $0_{*}(d)(x,y)\le d'(x,y)\iff d(x,y)\le\infty_{!}(d')(x,y)$
is the proof of the adjunciton. Since $\infty_{!}$ does not map the
bottom element of $W_{0}$ to the bottom element of $W_{\emptyset}$
we see that $\infty_{!}$ does not have a right adjoint, and this
is the end of the line for the adjunctions. 

For the property $\Sigma$, the functor $\Sigma^{*}:W_{\Sigma}\to W_{\emptyset}$
has both a left adjoint and a right adjoint. The left adjoint $\Sigma_{!}$
is given by $\Sigma_{!}(d)(x,y)=d(x,y)\vee d(y,x)$ while the right
adjoint $\Sigma_{*}$ is given by $\Sigma_{*}(d)(x,y)=d(x,y)\wedge d(y,x)$. 

\begin{lemma}
For any $P \subseteq \{0,\Delta,s\}$, $\Sigma_{*}$ (resp. $\Sigma_{!}$) has no right adjoint (resp. no
left adjoint).
\end{lemma}
\begin{proof}
This proof is constructed so as to satisfy any $P \subseteq \{0,\Delta,s\}$. We prove it for $|X|=2$
and extend it to any cardinality at the end of this proof.

$(\Sigma_{!})$ Assume as given and recall that if we can show that
there exsits a join in $W_{P}$ that $\Sigma_{!}$ doesn't preserve
then it doesn't have a right adjoint. Let $m,d\in W_{P}$ so that
$2>m(x,y)>d(y,x)>m(y,x)>d(x,y)>1$. Since we also want $m,d\in W_P$ for any $P$, we can make $m$ and $d$ so as to satisfy $\Delta, 0$ and $s$: clearly $\Delta$ and $s$ are satisfied and we let $m(x,x)=m(y,y)=d(x,x)=d(y,y) =0$. That
is, the condition presented above on their distances between $x$
and $y$ does not interfere with any axiom from $\{0,\Delta,s\}$,
only symmetry. Next, we show that
$$
\Sigma_{!}(m\wedge d)(x,y)<\left(\Sigma_{!}(m)\wedge\Sigma_{!}(d)\right)(x,y)
$$

since that proves that $\Sigma_{!}(m\wedge d)<\Sigma_{!}(m)                                                                                                                                                                                                                                                                                                                                                                                                                                                                                                                                                                                                                                                                                                                                                                                                                                                                                                                                                                                                                                                                                                                                                                                                                                                                                                                                                                                                                                                                                                                                                                                                                                                                                                                      \wedge\Sigma_{!}(d)$.
Notice that $(m\wedge d)(x,y)=d(x,y)$
and $(m\wedge d)(y,x)=m(y,x)$ (where $d\wedge m$ is given point-wise) and so $\Sigma_{!}(m\wedge d)(x,y)=\Sigma_{!}(m\wedge d)(y,x)=m(y,x)$.
Similarly, $\Sigma_{!}(m)(x,y)=\Sigma_{!}(m)(y,x)=m(x,y)$ and $ $$ $$\Sigma_{!}(d)(x,y)=\Sigma_{!}(d)(y,x)=d(y,x)$
and so $ $$\left(\Sigma_{!}(m)\wedge\Sigma_{!}(d)\right)(x,y)=d(y,x)>m(y,x)=\Sigma_{!}(m\wedge d)(x,y)$. 

This proof also works for any other $X$. The reason
we restricted the distances above between $1$ and $2$ is so that if
there are other points in $X$, then we can just make their distance
to $x$ and $y$ and between themselves $=1$.

$(\Sigma_{*})$ The above example also works for $\Sigma_{*}$; one must only show that it doesn't preserve that above
metrics' join and can also be extended to any other $X$
just as it was done with the left adjoint.
\end{proof}
Next we explore $\Delta^{*}:W_{P\cup\Delta}\hookrightarrow W_{P}$; it has only a right adjoint, $\Delta_{*}$, and we show
this right adjoint of $\Delta^{*}$ to have no right adjoint (in very
much the same spirit as with the previous property). 

\begin{lemma}
If $\Delta\not\in P$ and $\Delta^{*}:W_{P\cup\Delta}\to W_{P}$, then
$\Delta_{*}$ has no right adjoint.
\end{lemma}

\begin{proof}
To avoid degenerate cases we demand $|X| = 3$ and show how to extend it to any other cardinality. We create two weight
structures that will satisfy all axioms from $ \left\{ s,0,\Sigma\right\} $
so as to prove the statement for any $P$. Let $X=\left\{ x,y,z\right\} $ and $m,d\in W_{ \left\{ s,0,\Sigma\right\}}$ so
that all distances are bounded below by 1 and above by 2 and
\begin{itemize}

\item $m(x,y)+m(x,z)<m(y,z)$,

\item $d(x,y)+d(x,z)<d(y,z)$, and

\item $m(y,z)<m(x,z)+d(x,y)$, $m(x,y)<d(x,y)$, $m(x,z)>d(x,z)$ and $m(y,z)>d(y,z)$.
\end{itemize}
First notice that $(m\vee d)$ is taken
point-wise since $W_{ \left\{ s,0,\Sigma\right\}}$ is closed under non-empty joins from $W_\emptyset$. It follows that, by design,
$m\vee d\in W_{P\cup\Delta}$ and thus $\Delta_{*}(m\vee d)=m\vee d$.
Next, observe $\Delta_{*}(m)(x,y)=m(x,y)$, $\Delta_{*}(m)(x,z)=m(x,z)$
and $\Delta_{*}(m)(z,y)=m(x,y)+m(x,z)$; $\Delta_{*}(d)(x,y)=d(x,y)$,
$\Delta_{*}(d)(x,z)=d(x,z)$ and $\Delta_{*}(d)(z,y)=d(x,y)+d(x,z)$.
To this end we have 
\[
\Delta_{*}(m)(y,z)\vee\Delta_{*}(d)(y,z)<m(y,z)=\Delta_{*}(m\vee d)(x,y).
\]

Again, to extend this example to any cardinality, we just evaluate
all other distance to and from $x,y,z$ to other points in $X$ and
between all other points in $X$ to be $=1$.
\end{proof}

\section{Properties of $\psi_P: W_P \rightarrow Top(X)$}\label{sec:propsi}

For any $P\subseteq \{\Delta, \Sigma, 0\}$ we let $\psi^+_P: W_P \to Top(X)$ (resp. $\psi^-_P: W_P \to Top(X)$) so that $d$ is mapped to the topology its corresponding cover $\{B_L(x)\epsilon\mid x \in X \mbox{ and } \epsilon >0\}$ (resp. $\{B_R(x)\epsilon\mid x \in X \mbox{ and } \epsilon >0\}$) generates. It is clear that if $\Sigma \in P$ then $\psi^-=\psi^+$ and we assume that $0\in P$ throughout this section.

Several order-theoretic questions about each function $\psi_P^{+},\psi_P^{-}$ arise naturally. For instance, when are $\psi_P^{+},\psi_P^{-}$ order-preserving? Are meets and/or joins preserved? What can be said about the fibers of $\psi_P^{+},\psi_P^{-}$?

Notice that the above questions need only be solved by either one of any pair $\psi_P^{+},\psi_P^{-}$. Indeed, recall from Section~\ref{sec:met(X)} that the function $f:W_P \to W_P$ sending a generalized weight structure to its corresponding dual is an order isomorphism. In particular, we have that $\psi_P^{+} = f \circ \psi_P^{-}$. In view of this, we will only concern ourselves with $\psi_P^{+}$ and, for convenience, we denote it plainly by $\psi_P$.

We begin by noting that if the function $\psi_P$ is closed under any property from $\{$binary meets, binary joins, arb. meets, arb. joins$\}$, then so is each fiber $\psi_P^{-1}(\tau)$ for any $\tau \in Top(X)$. For instance, if $\psi_P$ is closed under binary joins, then take any pair $d,m\in \psi_P^{-1}(\tau)$ for some $\tau \in Top(X)$.  Since $\psi_P$ is closed under binary joins, then $\psi_P(d \vee m) = \psi_P(d)\vee \psi_P(m) = \tau \vee \tau$ which then implies that $d \vee m \in \psi_P^{-1}(\tau)$. Similar arguments prove the above for the remaining properties.

\begin{lemma}\label{lem:Pprop} For $P \subseteq \{\Delta,0,\Sigma\}$ and $M\in\{$binary meets, binary joins, arb. meets, arb. joins$\}$ if $\psi_P$ is closed under $M$, then each fiber of $\psi_P$ is a sub $M$-semillatice of $W_P$.
\end{lemma}

Incidentally, for $\Delta \in P$ $\epsilon$-balls form a base and the function $\psi_P$ is order-preserving and preserves finite non-empty joins (thus $\psi_P$ fibers are closed under binary joins). To see this, take $d$ and $m$ and denote $\psi_P(d) \vee \psi_P(m) = \tau$ and $p= m \vee p$. Since $p \geq m,d$ then $\psi_P(p)\geq \tau$. Let $\epsilon > 0$ and take any $y\in B_\epsilon^{d}(x) \cap B_\epsilon^{m}(x)$. Since $p(x,y) = \text{max}\{m(x,y),d(x,y)\}$ then $p(x,y) < \epsilon$. It follows that $y \in B_\epsilon^{p}(x)$, $\psi_P(p) = \tau$ and thus $\psi$ preserves joins.
Next, take $d \in W_P$ so that $d(x,y) = 0$ or $\infty$ for all $x,y \in X$; the topology generated by $d$ partitions $X$ into basic open sets (i.e. open sets where $d(x,y) = 0$ for all $x,y$ in the open set). Notice that if $d'> d$ then $\psi_P(d') > \psi_P(d)$ and, thus, $d = \bigvee(\psi_P^{-1} \circ \psi_P (d)) = \bigvee \{d' \in W_P \mid \psi_P(d) = \psi_P(d')\}$. Consequently, $\psi_P^{-1}(d)$ is closed under non-empty joins. Call any such weight structure (i.e., one whose evaluations are all either $0$ or $\infty$) a \emph{partition weight structure}.

\begin{theorem}\label{thm:met} For $\Delta \in P \subseteq \{\Delta,0,\Sigma\}$
\begin{enumerate}[(a)]
\item
 \begin{enumerate}[(i)]
\item $\psi_P$ preserves binary joins, and fibers along $\psi_P$ are join sublattices of $W_P$.
\medskip
\item With the exception of topologies generated by partition weight structures, $\psi_P$ fibers are not closed under non-empty joins. Consequently, $\psi_P$ does not preserve arbitrary joins.
\end{enumerate}
\medskip
\item In general, $\psi_P$ fibers are not lattices and $\psi_P$ does not preserve binary meets.
\end{enumerate}

\end{theorem}

\begin{proof} $(a)(i)$ has been proved already. For $(a)(ii)$ it is clear that fibers that contain a partition weight structure are closed under non-empty joins. Otherwise, if for some $\tau$ whose fiber does not contain a partition weight structure we have $d = \bigvee \psi_P^{-1}(\tau) \in \psi_P^{-1}(\tau)$, then $d<2d \in \psi_P^{-1}(\tau)$ (a contradiction). By Lemma \ref{lem:Pprop}, $\psi$ does not preserve arbitrary joins.
For $(b)$ we make use of Figure~\ref{fig:deltameet}, where all points are taken to depict points from the plane and the dashed lines represent concentric arcs centered at $x$ (the radii of these arcs converges to $0$). Let $X_1$ (resp. $X_2$) denote the collection $\{x\}\cup\{y_i\mid 1\leq i\}$ in addition to all red sequences (resp. blue sequences). We also let $d:W_P(X_1)$ and $m:W_P(X_1)$ be the usual planar metrics on their respective sets of points. Notice that for any fixed $j\geq 1$ neither $X_1$ nor $X_2$ contains the limit of any $\{j_n\}_{n\in\N}$ sequence and the same is true of the $\{y_n\}_{n\in\N}$ sequence for both sets. The obvious function $F:X_1\to X_2$ where $F(y_i)=y_i$, $F(x)=x$ and all red points are mapped to their corresponding blue points generates a homeomorphism between $(X_1,d)$ and $(X_2,m)$. The claim is true on all points other than $x$ itself (since both metrics are discrete on all points other than $x$). For $x$ and any $\delta >0$ if $\delta$ is the same as the radius of any one of the dashed arcs, then $B_\delta^{d}(x) = B_\delta^{m}(f(x))$. Let $c_1 < \delta < c_2$ where $c_1$ and $c_2$ represent the radii of a pair of adjacent arcs. Obviously, $B_\delta^{d}(x) \subseteq B_\delta^{m}(x)$ and since $\{y\}$ ($y\not=x$) is open in $d$ then $B_\delta^{m}(x) = B_\delta^{d}(x) \cup \{y\mid f(y)\in B_\delta^{m}(x)\}$. Hence, both metrics generate homeomorphic topologies, in which case we will consider $X_1$ and $X_2$ as being the same set and refer to it as $X$.

\begin{figure}[H]
\includegraphics[scale = .5]{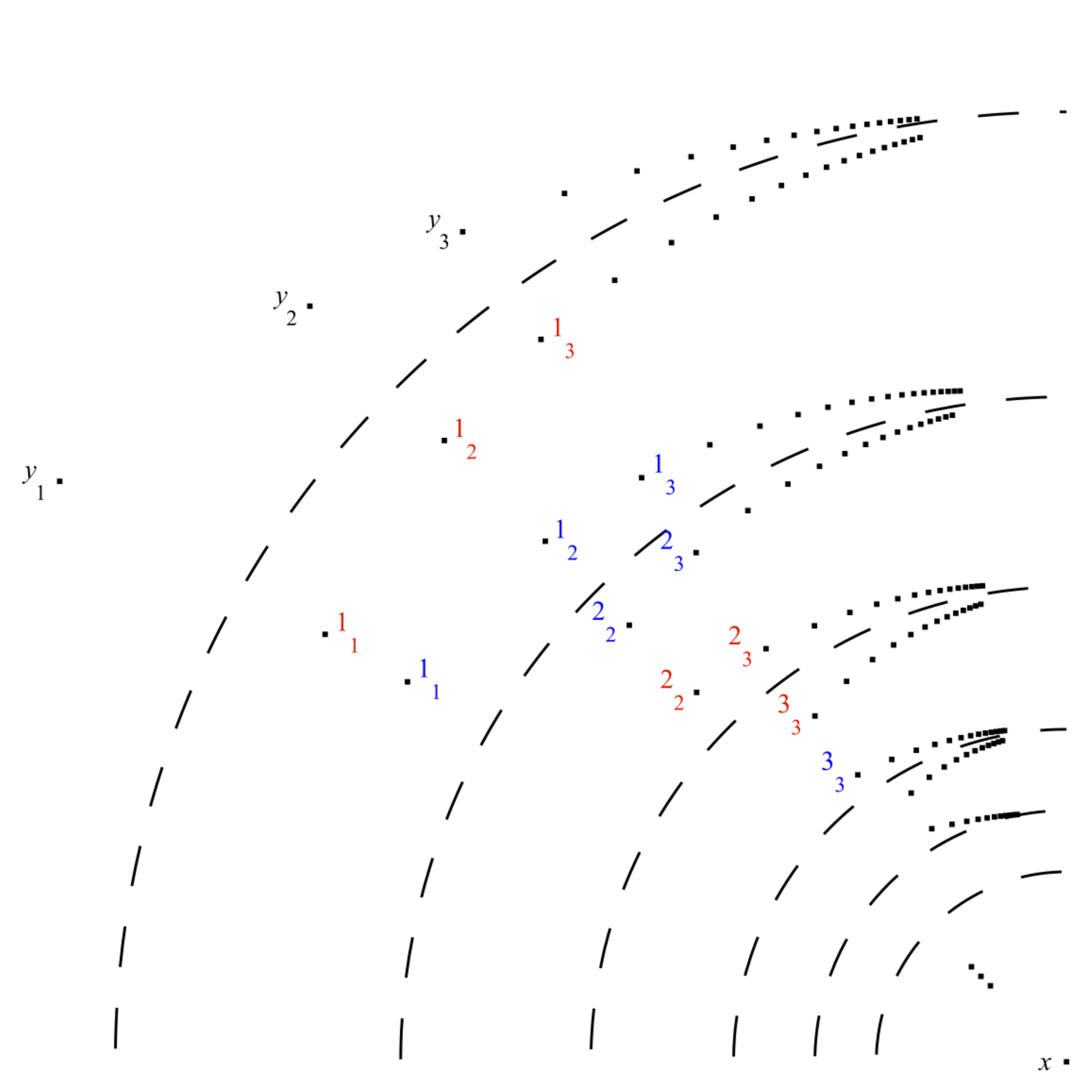}
\caption{The metrics $d$ and $m$}
\label{fig:deltameet}

\end{figure}


 Let $p=d \wedge m$: we will show that $\{y_n\}_{n\in\N}$ converges to $x$ in $(X,p)$. For $p$, the shortest way to travel from $x$ to $y_1$ is to go from $x$ to $1_1$ via $m$ and from $1_1$ to $y_1$ via $d$. The same happens between $x$ and $y_2$ (i.e. $p(x,y_2) = d(x,2_2)+m(2_2,1_2)+d(1_2,y_2))$. The sequences are taken to represent sets of points that become arbitrarily close to each other. For instance, for the blue $\{1_i\}$ sequence and the blue $\{2_i\}$ sequence, we demand that $\displaystyle\lim_{j\to \infty} m(1_j,2_j) = 0$. Thus $\displaystyle \lim_{n\to \infty} p(x,y_n) = 0$ and the sequence $\{y_n\} \xrightarrow{p} x$. This tells us that there is a convergent sequence in $(X,p)$ that converges in neither $(X,d)$ nor $(X,m)$ and thus the topology generated by $p$ is strictly coarser than the one generated by $d$ and $m$.
\end{proof}







The scenarios where $\Delta \not \in P$ are not as clear since $\epsilon$-balls do not form a base of open sets; $\psi_P$ is actually a function from $W_P$ into $Cov(X)$. In particular, as the following example shows, $\psi_P$ in general is not order-preserving.

\begin{example}\label{exa:notpres} Let $X = \{x_n\} \cup \{y_n\}\cup \{x\}$ and $d, m \in W_{0,\Sigma}$ so that
for all $a,b \in X$
\[
d(a,b)=\begin{cases}
0 & a,b\in \{x_n\} \cup \{y_n\}\\
2^{-n} & a=x,b=x_n \mbox{ or } a=x,b=y_0\\
\frac{d(x,x_n) + d(x,x_{n+1})}{2} & a=x,b=y_n,n>0.
\end{cases}
\]
and
\[
m(a,b)=\begin{cases}
0 & a,b\in \{x_n\} \cup \{y_n\}\\
2^{-n} & a=x,b=x_n \mbox{ or } a=x,b=y_0\\
\frac{d(x,x_{n-1}) + d(x,x_{n})}{2} & a=x,b=y_n,n>0.
\end{cases}
\]\\

Clearly, $m > d$ and neither of the topologies generated by them contains the other.
\end{example}

Example~\ref{exa:notpres} also shows that neither $\psi_0$ nor $\psi_{0,\Sigma}$ are closed under either meets or joins. Indeed, $\psi_P(m \vee d)= \psi_P(m)<\psi_P(d)\vee \psi_P(m)$ and $\psi_P(m \wedge d)= \psi_P(d)>\psi_P(d)\wedge \psi_P(m)$. Showing that fibers are not sublattices is more involved.

\begin{theorem}\label{thm:3part} For $\Delta \not \in P$
\begin{enumerate}[(a)]
\item $\psi_P$ is not order-preserving.
\medskip
\item $\psi_P$ preserves neither meets nor joins.
\medskip
\item Fibers along $\psi_P$ are, in general, not closed under meets in $W_P$.
\end{enumerate}

\end{theorem}

\begin{proof} Example~\ref{exa:notpres} proved part $(a)$, and $(a) \Rightarrow (b)$. We prove $(c)$ for the case where $P = \{0,\Sigma\}$ and note that the scenario where $P=\{0\}$ follows immediately. Take $X = \{x_n\} \cup \{y,z\}$ (all distinct points) and define $d,m \in W_{0,\Sigma}$, with $0<\alpha<1$, as follows

\[
d(a,b)=\begin{cases}
\alpha & a = y,b=z \\
2^{-n} & a=y,b=x_n\\
1 & a=z,b=x_n\\
|d(y,x_i)-d(y,x_j)| & a=x_i, b=x_j.
\end{cases}
\]

and

\[
m(a,b)=\begin{cases}
d(a,b) & a = y,b=z\\
d(y,x_n) & a=z,b=x_n\\
d(z,x_n) & a=y,b=x_n\\
|d(y,x_i)-d(y,x_j)| & a=x_i, b=x_j.
\end{cases}
\]

\medskip

Notice that both weight structures generate the same topology on $X$. In particular, the set $\{y,z\}$ can be separated from the sequence $\{x_n\}$. Next, let $p = d\wedge m$ and notice that in the topology generated by $p$ we have $\{x_n\}\rightarrow \{y,z\}$. Indeed, in $p$ any $\epsilon$-ball about either $z$ or $y$ contains an infinite tail of $\{x_n\}$ and, thus, the cover generated by $p$ cannot separate $\{y,z\}$ from the sequence.
\end{proof}

Weight structures that do not satisfy the triangle inequality are usually referred to as \textit{premetrics} and \textit{semimetrics} (depending on whether or not they satisfy symmetry). A standard way of generating a topology based on a premetric (in contrast to generating a cover that then generates a topology) is as follows

\begin{definition} The function $\phi_P:W_P \to Top(X)$ is the one for which given $d\in W_P$ then $O\in \phi_P(d)$ iff for all $x \in O$ we can find $\epsilon > 0$ so that $B_\epsilon(X) \subseteq O$.
\end{definition}

The following result can be easily verified.

\begin{lemma} For any $P\subseteq \{0,\Sigma,\Delta\}$ the function $\phi_P$ is order-preserving. Moreover, given any $d\in W_P$, $\psi_P(d) \geq \phi_P(d)$.
\end{lemma}

 Notice that if $\Delta \in P$, $\psi_P = \phi_P$. Otherwise, even with this stronger definition of an open set, $\epsilon$-balls are not guaranteed to be open sets.

\begin{theorem} For $0\in P$
\begin{enumerate}[(a)]
\item
 \begin{enumerate}[(i)]
\item $\phi_P$ preserves binary meets, and fibers along $\phi_P$ are sub meet semilattices of $W_P$.
\medskip
\item Fibers along $\phi_P$ are not closed under binary joins, and $\phi_P$ does not preserve joins.
\end{enumerate}
\medskip
\item In general, $\phi_P$ fibers are not closed under arbitrary meets, and $\phi_P$ does not preserve arbitrary meets.
\end{enumerate}
\end{theorem}

\begin{proof} Theorem~\ref{thm:met} deals with the case where $\Delta \in P$ and we assume that $\Delta \not \in P$ throughout the proof. In particular, this means that given a pair of weight structures from $ W_P$ then their meet is taken pointwise. \\
For $(a)(i)$ first notice that since $\phi_P$ is order-preserving it follows that for any pair $d,m \in W_P$ we have $\phi_P(d \wedge m) \leq \phi_P(d) \wedge \phi_P(m)$. Next, let $O \in  \phi_P(d) \wedge \phi_P(m)$ and denote $p = d \wedge m$. In turn, for any $x\in O$ we can find  $\epsilon_1,\epsilon_2 > 0$ so that $B^{d}_{\epsilon_1}(x),B^{m}_{\epsilon_2}(x) \subseteq O$. Letting $\epsilon = \text{min}\{\epsilon_1,\epsilon_2\}$ it is simple to verify that $B^{p}_{\epsilon} = B^{d}_{\epsilon} \cup B^{m}_{\epsilon} \subseteq B^{d}_{\epsilon_1}(x) \cup B^{m}_{\epsilon_2}(x) \subseteq O$. Hence, $\phi_P(p) = \phi_P(d) \wedge \phi_P(m)$.\\
For $(a)(ii)$ we make use of part of the proof of Theorem~\ref{thm:3part} where $X=\{x_n\}\cup\{z,y\}$ as follows

\[
d(a,b)=\begin{cases}
0 & a = y,b=z\\
2^{-n} & a=y,b=x_n\\
1 & a=z,b=x_n\\
|d(y,x_i)-d(y,x_j)| & a=x_i, b=x_j.
\end{cases}
\]
and
\[
m(a,b)=\begin{cases}
d(a,b) & a = y,b=z\\
d(y,x_n) & a=z,b=x_n\\
d(z,x_n) & a=y,b=x_n\\
|d(y,x_i)-d(y,x_j)| & a=x_i, b=x_j.
\end{cases}
\]\\

Clearly, $\phi_P(d \vee m)$ contains $\{y,z\}$ as an open set and both weight structures generate the same topology on $X$. Since any open set containing $y$ in $\phi_P(d)$ contains an infinite tail of the sequence $\{x_n\}$ then $\phi_P(d)$ cannot separate $y$ from the sequence  (the same is true for $z$ w.r.t. $m$). Consequently, the topology generated by $d\vee m$ is strictly larger than the one generated by $\phi_P(d) \vee \phi_P(m)$. It is important to notice that (just as with Theorem~\ref{thm:3part}) the non-symmetric case follows immediately. \\
That last claim is simpler: take a two-point set $X= \{x,y\}$ and $d_n \in W_P$ so that $d_n(x,y) = \frac{1}{n}$. Each weight structure generates the discrete topology on $X$ but their meet generates the trivial topology.
\end{proof}

\section{$Met(X)$}\label{sec:metx}
\subsection{Lattice Properties}

Recall that for a set $X$, $Met(X) = W_{0,\Delta,\Sigma} = W_{0} \cap W_\Delta \cap W_\Sigma$ and since all of the latter objects are sub join complete semilattices of the ambient lattice, so is $Met(X)$. In the previous section we constructed an instance where two elements from $W_\Delta$ have a meet outside $W_\Delta$. Furthermore, the construction allows for the weight structures to be in $Met(X)$ as long as $|X| >2$. Hence, $Met(X)$ is not a sublattice of $W_\emptyset$. Of course, since $Met(X)$ is a join complete lattice, it is a complete lattice in its own right. The meet operation is different to that of the ambient lattice; for a collection $\{d_i\}_{i\in I} \subset Met(X)$ the meet $d$ (within $Met(X)$) is given by 

$$
d(x,y) = \displaystyle \bigwedge_{\gamma:(x=y_0,\ldots,y_n = y)} \sum_j \left( \bigwedge_i d_i(y_j,y_{j+1})\right),
$$
\smallskip
\noindent
for any pair $ x,y \in X$. 

It is not hard to see that $Met(X)$ does not contain atoms (resp. anti-atoms); simply choose any element $d \in Met(X)$ and let $d' \in Met(X)$ be the one that for all $x,y \in X$ yields $d'(x,y) = \frac{d(x,y)}{2}$ (resp. $d'(x,y) = 2d(x,y)$). However, there exists a very special collection of metric structures that closely resembles anti-atoms in $Met(X)$.

\begin{definition} For $x\in X$ and $\alpha\in(0,\infty]$ let
 $d^{\alpha,x,y} \in Met(X)$ with $d_{\alpha,x,y}(x,y)=\alpha$ and, $d_{\alpha,x,y}(z,w)=\infty$ otherwise. We will be refer to these objects as \emph{pseudo-anti-atoms}.
\end{definition}

Pseudo-anti-atoms are, is a sense, the collection of the \emph{largest} elements from $Met(X)$. That is, for any metric structure $d \in Met(X)$ we can find a pseudo-anti-atom, call it $d^{\alpha,x,y}$, so that $d<d^{\alpha,x,y}<d_\infty$. In particular, we can describe any metric structure as a meet of pseudo-anti-atoms; take $d\in Met(X)$ and notice that

$$
\displaystyle d= \bigwedge_{x,y\in X} d^{d(x,y),x,y}.
$$

In the literature on $Top(X)$ one also looks
at basic intervals \cite{MR0305322} (i.e., chains where each link is obtained
by a minimal change).

\begin{definition}
For $d,m\in Met(X)$ write $d\prec_n m$
if $d<m$ and if there exists a finite collection $\{x_j \mid 1 \leq j \leq n\}$ of distinct elements from $X$ such that $d(u,v)=m(u,v)$
for all $u,v\in X$ so that $\{u,v\}\not \subseteq \{x_j\}$.
For $d,m\in Met_1(X)$ we say that the interval $[d,m]$
is an \emph{$n$-elementary interval }if $\forall d_{3},d_{4}\in[d,m]$
holds that if $d_{3}<d_{4}$ then $d_{3}\prec_n d_{4}$. Say that $d$
is \emph{$n$-elementarily maximal} if no $d'$ exists with $d\prec_n d'$ and \emph{$n$-elementarily minimal} if no $d'$ exists so that $d' \prec_n d$.
\end{definition}

If $d \prec_n m$ then $d \prec_m m$ for all $m\geq n$. Consequently, an $n$-elementary interval (resp. a metric structure is $n$-elementary maximal, $n$-elementary minimal) is $m$-elementary (resp. $m$-elementary maximal, $m$-elementary minimal) for all $m\geq n$ (resp. $m\leq n$). In Section~\ref{sec:embed} we shall prove that any finite lattice embeds in an $n$-elementary interval for some $n \in \N$.

\begin{lemma} For any $n\in \N$ if $d \prec_n m$ with $0< d(x,y)$ $\forall x,y\in X$ and $d(x,y) < m(x,y)$, then

\begin{enumerate}[(a)]
\item The metric structures $d$ and $m$ generate the same topology on $X$.
\item All metric structures in $[d,m]$ generate the same topology on $X$.
\end{enumerate}
\end{lemma}
\begin{proof} Clearly $(a)$ iff $(b)$ and we prove $(a)$. We begin by letting $\{x_j\}_{j\leq n}$ be the set for which $d(x_i,x_j)<m(x_i,x_j)$ ($i\ne j$). All open balls about any $z\in X\smallsetminus \{x_j\}_{j\leq n}$ are the same for both metric structures. Also, any $m$ $\epsilon-$ball about any $x_i$ is finer that a $d$ $\epsilon-$ball about $x_i$. Lastly, for any $\epsilon > 0$ let $0<\delta <\text{min}\{m(x_i,x_j)\mid i\ne j\}$ and notice that $B_\delta^{d}(x) \subseteq B_\epsilon^{m}(x)$. Since both metric structures generate the same topology the proof is complete.
\end{proof}

Recall that a metric space $(X,d)$ is \emph{Menger convex }if for
all $x\ne y\in X$ and $0<r<d(x,y)=L$ there exists a point $p\in X$
satisfying $d(x,p)=r$ and $d(p,y)=L-r$. Similarly, we define the dual of Menger convexity: a metric
space $d$ is \emph{Menger$^{*}$ convex} provided for any pair $x\ne y\in X$ there
exists a $z\in X$ so that $d(y,z)=d(x,y)+d(x,z)$. We adopt both definitions for the collection $Met(X)$. 

\begin{theorem}
A metric structure $d$ is $1$-elementary maximal (resp. $1$-elementary minimal) if, and only if, $\forall x,y\in X$ and
all $\epsilon>0$ there exists a $z\in X$ so that $d(x,y)>d(x,z)+d(y,z)-\epsilon$ (resp. $d(x,z)>d(x,y)+d(y,z)-\epsilon$).
\end{theorem}

\begin{proof} We prove the claim for elementary maximality since its dual follows immediately. For sufficiency of any such $d$, it is clear that if for any pair $x\not = y \in X$ and an arbitrary $\epsilon >0$ we can find $z\in X$ so that $d(x,y)>d(x,z)+d(y,z)-\epsilon$, then it is impossible to find any $m$ so that $ d \prec_1 m$. Indeed, any such $m$ would violate the triangle inequality by only extending the distance of a distinct pair $x,y$ and nothing else. Necessity is also straight forward; the only axiom that inhibits a 1-elementary extension on any such $d$ is the triangle inequality. 
\end{proof}

The following is simple to verify.

\begin{corollary}
Any Menger (resp. Menger$^{*}$) metric structure is $1$-elementarily maximal (resp. $1$-elementarily minimal).
\end{corollary}

\subsection{Lattice Embeddability}\label{sec:embed}

In \cite{MR0016750} Whitman proved that any lattice can be embedded in the lattice of equivalence relations, $Eq(X)$, for some set $X$. We exploit this remarkable result in the following where $\psi: Met(X) \to Top(X)$ and $\top$ is the discrete topology on $X$.

\begin{theorem}\label{thm:anylattice}
Any lattice can be embedded in $\psi^{-1}(\top)$ for some set $X$.
\end{theorem}
\begin{proof} Take $1<\alpha<2$ and let $\phi:Eq(X)\rightarrow Met(X)$ so that ${\sim}\in Eq(X)$
 $\phi(\sim) = d_\sim \in Met(X)$ where

\[
d_\sim(a,b)=\begin{cases}
\alpha & \text{ if }  a\sim b \text{ and}\\
1 & \mbox{otherwise.}
\end{cases}
\]\\

Then for $\{\sim_{i}\}_{i\in I} \subset Eq(X)$ and ${\sim}=\bigwedge\sim_{i}$
we have $a\sim b$ if and only if $(a,b)\in{\sim_{i}}$ for all $i.$
The same happens with $\phi(\sim)=\phi(\bigwedge{\sim_{i}})$. That is, $\phi(\sim)(a,b)=\alpha$
if and only if $\phi(\sim_{i})(a,b)=\alpha$ for all $i.$ Hence,
$\phi(\sim)=\bigwedge\phi(\sim_{i})$. The same occurs with joins;
for ${\sim} =\bigvee\sim_{i}$ then $\phi(\sim)=\bigvee\phi(\sim_{i})$. Notice that the above joins and meets in $\psi^{-1}(\top)$ are in agreement with the meets and joins from the ambient lattice.
Thus we have embedded $Eq(X)$ within $\psi^{-1}(\top)$.
\end{proof}

The authors of \cite{MR1617095} show that all finite distributive lattices occur as intervals between Hausdorff topologies; we show something similar occurs in $Met(X)$. Until 1980, one of the outstanding questions in lattice theory was a conjecture by Whitman: \emph{every finite lattice can be embedded in some $Eq(X)$, for a finite set $X$.} The conjecture was turned into a theorem by P. Pudl\'{a}k and J. T\r{u}ma in \cite{finitelattice}.

\begin{theorem} Any finite lattice can be embedded in an $n$-elementary interval within $Met(X)$ for some finite set $X$.
\end{theorem}

\begin{proof} We will embed $Eq(X)$ in an $n$-elementary interval, where $n = \frac{|X|(|X|-1)}{2}$. Let $d, d_\alpha \in Met(X)$ so that for all $x,y \in X$, $d(x,y) = 1<d_\alpha(x,y) = \alpha<2$. In particular, $d \prec_n d_\alpha$ and $[d,d_\alpha]$ is $n$-elementary. Just as with Theorem~\ref{thm:anylattice} we let $\phi:Eq(X)\rightarrow Met(X)$ so that ${\sim}\in Eq(X)$
 $\phi(\sim) = d_\sim \in Met(X)$ where

\[
d_\sim(a,b)=\begin{cases}
\alpha & \text{ if }  a\sim b \text{ and}\\
1 & \mbox{otherwise.}
\end{cases}
\]\\
Since any lattice embeds in the lattice $Eq(X)$ for some finite set $X$ and we've embedded $Eq(X)$ in an $n$-elementary interval, the proof is complete.

\end{proof}

\bibliographystyle{plain}
\bibliography{mybib2}

\end{document}

Next we define some covers for $X$ based on the following sets where $x\in X$ and $\epsilon>0$:

$$B_{L,\Delta}^{d}(x)_{\epsilon}=\{y\in X\mid\exists\mbox{ }\gamma:(x=y_{0},\ldots,y_{n}=y)\mbox{ with }\sum d(y_{i},y_{i+1})<\epsilon\},$$
 $$
 B_{R,\Delta}^{d}(x)_{\epsilon}=\{y\in X\mid\exists \mbox{ }\gamma:(y=y_{0},\ldots,y_{n}=x)\mbox{ with }\sum d(y_{i},y_{i+1})<\epsilon\},$$

$$ \quad B_L^{d}(x)_\epsilon=\{y\in X\mid d(x,y)<\epsilon\},$$

$$ \quad B_R^{d}(x)_\epsilon=\{y\in X\mid d(y,x)<\epsilon\}.$$
\smallskip

Let $d_L$ and $d_R$  be the collections $\{B_{L,\Delta}^{d}(x)_{\epsilon}\mid x\in X\mbox{ and }\epsilon>0\}$ and $\{B_{R,\Delta}^{d}(x)_{\epsilon}\mid x\in X\mbox{ and }\epsilon>0\}$ respectively. With the above in mind, define $W_0 \to Cov \times Cov$ so that $d \mapsto (d_L,d_R)$.
  As for $W_{0,\Delta}\rightarrow Cov \times Cov$, we let $d\mapsto(\{B_L^{d}(x)_\epsilon\}, \{B_R^{d}(x)_\epsilon\})$.

If, for instance, for some $d\in W_{0}$
we have $y\in B_{L,\Delta}^{d}(x)_\epsilon$, then clearly $G(d)(x,y)<\epsilon$
(where $G$ is left adjoint to $W_{0,\Delta} \hookrightarrow W_0$) and thus $y\in B_{L}^{G(d)}(x)_{\epsilon}$. Similarly, if $y\in B_{L}^{G(d)}(x)_{\epsilon}$
then $G(d)(x,y)<\epsilon$ and by definition $y\in B_{L,\Delta}^{d}(x)_\epsilon.$
Consequently, for any $d\in W_0$ we get that

\begin{figure}[H] 
\vspace{.2cm}
\centering 
\def\svgwidth{100pt} 
\input{Diagram3.pdf_tex}
 \end{figure}
commutes. 

In turn, since $Top \hookrightarrow Cov$ is an adjunction

\begin{figure}[H] 
\vspace{.2cm}
\centering 
\def\svgwidth{100pt} 
\input{Diagram4.pdf_tex}
 \end{figure}

commutes as well.

For any $d\in W_\emptyset$, consider the following candidates for an $\epsilon$-ball:\\

\begin{itemize}
\item $B_{L}(x)_{\epsilon}=\{y\in X\mid d(x,y)<\epsilon\}$, $B_{R}(x)_{\epsilon}=\{y\in X\mid d(y,x)<\epsilon\}$
\smallskip
 and $B_{\Sigma}(x)_{\epsilon}=B_{R}(x)_{\epsilon}\cup B_{L}(x)_{\epsilon}$.
\medskip
\item $B_{0,L}(x)_{\epsilon}=B_{L}(x)_{\epsilon}\cup{\{x\}}$ and $B_{0,R}(x)_{\epsilon}=B_{R}(x)_{\epsilon}\cup{\{x\}}$
\medskip
\item $B_{\Delta,L}(x)_{\epsilon}=$
$$\left\{ y\in X\mid{\exists}{\gamma:(x=y_{0},\ldots,y_{n}=y)} \mbox{ with }\sum d(y_{i},y_{i+1})<\epsilon\right\} $$
\item $B_{\Delta,R}(x)_{\epsilon}=$
$$\left\{ y\in X\mid{\exists}\gamma:(y=y_{0},\ldots,y_{n}=x) \mbox{ with }\sum d(y_{i},y_{i+1})<\epsilon\right\} $$
\item $B_{\Sigma,0}(x)_{\epsilon}=B_{\Sigma}(x)_{\epsilon}\cup{\{x\}}=B_{0}(x)_{\epsilon}\cup B_{R}(x)_{\epsilon}$
\medskip
\item $B_{\Delta,\Sigma}(x)_{\epsilon}=$
$$\left\{ y\in X\mid\exists\gamma:(x=y_{0}\ldots,y_{n}=y) \mbox{ with }\sum\text{min\{}d(y_{i+1},y_{i}),d(y_{i},y_{i+1})\}<\epsilon\right\} $$
\item $B_{\Delta,0,L}(x)_{\epsilon}=B_{\Delta,L}(x)_{\epsilon}\cup\{x\}$, $B_{\Delta,0,R}(x)_{\epsilon}=B_{\Delta,R}(x)_{\epsilon}\cup\{x\}$
\smallskip
and $B_{\Delta,\Sigma,0}(x)_{\epsilon}=B_{\Delta,\Sigma}(x)_{\epsilon}\cup\{x\}$\\
\end{itemize}


\begin{lemma} For $Q\subset P \subseteq \{0,\Delta,\Sigma\}$ and $G$ right adjoint to $W_P\hookrightarrow W_Q$ we have:

\begin{enumerate}[(a)]
\item If $S = P\smallsetminus Q$, then $B_{S,L}^d(x)_\epsilon = B_{L}^{G(d)}(x)_\epsilon$ and $B_{S,R}^d(x)_\epsilon = B_{R}^{G(d)}(x)_\epsilon$.
\item If $S = P\smallsetminus Q$, then

\begin{figure}[H] 
\vspace{.2cm}
\centering 
\def\svgwidth{100pt} 
\input{Diagram5.pdf_tex}
 \end{figure}
 commutes, provided for any $d\in W_Q$, $d \mapsto (\{B_{S,L}^d(x)_\epsilon\}, \{B_{S,R}^d(x)_\epsilon\})$ and $G(d) \mapsto(\{B_{L}^{G(d)}(x)_\epsilon\}, \{B_{R}^{G(d)}(x)_\epsilon\})$.
 \end{enumerate}
 \end{lemma}

 \begin{proof} Clearly $(a) \Rightarrow (b)$, and $(a)$ is clear whenever $S=\{0\}$ since we are only demanding for any $\epsilon-$ball about any point to contain the point itself. The example involving $W_0$ and $W_{0,\Delta}$ deals with the case where $S = \{\Delta\}$. For $S = \{\Sigma\}$ if we take any $d \in W_Q$ then for all $x,y \in X$, by design $G(d)(x) = \text{min}\{d(x,y) , d(y,x)\}$ and thus $B_{\Sigma}^d(x)_{\epsilon}=B^d_{R}(x)_{\epsilon}\cup B^d_{L}(x)_{\epsilon} = B^{G(d)}_\epsilon(x)$. Next, notice that $S = \{\Delta,0\}, \{\Sigma,0\}$ are immediate consequences of $S=\{\Delta\},\{0\}$ and $\{\Sigma\}$. To this end, we prove $(a)$ for $S=\{\Delta,\Sigma\}$ (and, without loss of generality, for only $B_{S,L}^{d}(x)_\epsilon$) and notice that $S=\{\Delta,\Sigma,0\}$ follows from it. Take any $d \in W_Q$. Since $d\geq G(d)$ then $B_{S,L}^d(x)_\epsilon \subseteq B_L^{G(d)}(x)_\epsilon = B_R^{G(d)}(x)_\epsilon = B_\epsilon^{G(d)}(x)$; symmetry ensures that right balls and left balls coincide with the standard balls (i.e., $B_\epsilon^{G(d)}(x)$). If $y \in B_\epsilon^{G(d)}(x) $ and since
 
 $$\displaystyle \epsilon >G(d)(x,y) = \bigwedge_{\gamma:(x=y_0,\ldots , y_n = y)} \sum_i \text{min}\{d(y_{i+1},y_{i}),d(y_{i},y_{i+1})\}$$
it follows that for some path $\gamma:(x=y_0, \ldots , y_n=y)$ we have
  
  $$ \sum_i \text{min}\{d(y_{i+1},y_{i}),d(y_{i},y_{i+1})\} < \epsilon$$
   and thus $y\in B_{S,L}^d(x)_\epsilon$.

\end{proof}

\begin{figure}[H] 
\vspace{.2cm}
\centering 
\def\svgwidth{300pt} 
\input{Diagram2.pdf_tex}
\vspace{.1cm}
\[
 \text{Left adjoints}
 \]
 \end{figure}
 
 \begin{figure}[H] 
\vspace{.2cm}
\centering 
\def\svgwidth{150pt} 
\input{BigDiagram.pdf_tex}
 \end{figure}